\documentclass[12pt]{amsart}
\usepackage{amscd,amssymb,epsfig}
\usepackage{graphics}

\input epsf
\newtheorem{theorem}{Theorem}[section]
\newtheorem{lemma}[theorem]{Lemma}
\newtheorem{proposition}[theorem]{Proposition}
\newtheorem{corollary}[theorem]{Corollary}
\theoremstyle{definition}
\newtheorem{definition}[theorem]{Definition}

\theoremstyle{remark}
\newtheorem{observation}[theorem]{Observation}

\numberwithin{figure}{section}
\numberwithin{table}{section}

\newcommand{\ra}{\rightarrow}
\newcommand{\hra}{\hookrightarrow}

\newcommand{\Teich}{Teichm\"uller }

\def\ra{{\rightarrow}}

\def\bZ{{\mathbb Z}}

\def\S1g{{\Sigma_{g,1}}}
\def\T1g{{\mathcal T}_{g,1}}
\def\M1g{{MC}_{g,1}}

\def\I1g{{\mathcal I}_{g,1}}

\newcommand\Aut{\operatorname{Aut}}

\newcommand\Pt{\mathfrak{Pt}(\S1g)}

\newcommand\Cs{\mathfrak{Cs}(\S1g)}

\newcommand\dFGC{\mathcal{\hat G}_T}
\newcommand\grdFGC{|\mathcal{\hat G}_T|}
\newcommand\tgrpd{\mathfrak{P}(|\mathcal{\hat G}_T|)}

\def\iprec{\curlyeqprec}
\def\ipred{\eqslantless}

\addtocounter{MaxMatrixCols}{5}

\newcommand{\mb}{\mathbf}

\begin{document}

\title[Chord Diagrams and the Mapping Class Group]
{A Chord Diagrammatic Presentation of the Mapping Class Group of a Once Bordered Surface}

\author{Alex James Bene}
\address{Department of Mathematics\\
Aarhus University\\
DK-8000 Aarhus C, Denmark\\
~{\rm and}~Department of Mathematics\\
University of Southern California\\
Los Angeles, CA 90089\\
USA\\}
\email{bene{\char'100}math.usc.edu}


\begin{abstract}

The Ptolemy groupoid is a   combinatorial groupoid generated by elementary moves on marked trivalent fatgraphs with three types of relations.  Through the fatgraph decomposition of \Teich space, the Ptolemy groupoid is a mapping class group equivariant  subgroupoid of the fundamental path groupoid of \Teich space with a discrete set objects.  In particular, it leads to an infinite, but combinatorially simple, presentation of the mapping class group of an orientable surface.  

  In this note,  we give a presentation of a full mapping class group equivariant  subgroupoid of the Ptolemy groupoid of an orientable  surface with one boundary component in terms of marked linear chord diagrams, with chord slides as generators and  five types of relations.    We also introduce a dual version of this presentation which has  advantages for certain applications, one of which is given.  
\end{abstract}

\maketitle


\section{Introduction}

In \cite{penner}, Penner introduced a combinatorial presentation of the mapping class group of a punctured surface $\Sigma$ in terms of elementary moves, called Whitehead moves, on fatgraphs.  This lead to the introduction of the Ptolemy groupoid of $\Sigma$, a discrete combinatorial version of the fundamental path groupoid of the \Teich space of $\Sigma$.     The quotient of this groupoid under the action of the mapping class group gives a discrete combinatorial version of the (orbifold) fundamental path groupoid of the moduli space of (punctured) surfaces.  

Later, Morita and Penner \cite{moritapenner} used the notion of homology marked fatgraphs to describe the (discrete combinatorial version of the) fundamental path groupoid of the Torelli space of a once-punctured surface $\Sigma_{g,*}$, which is the quotient of the \Teich space of $\Sigma_{g,*}$ by the Torelli group, the subgroup of the mapping class group of $\Sigma_{g,*}$  which acts trivially on the homology of $\Sigma_{g,*}$.  Moreover, they gave a finite presentation of  this Torelli groupoid as well as  the first (infinite) presentations of the Torelli groups.  

One  advantage of these groupoid versions of the mapping class  and Torelli groups  is that they provide  an alternate avenue for defining and describing representations of these  groups.  Indeed, Morita and Penner were able to use the Torelli groupoid to give a simple combinatorial lift of the classical Johnson homomorphism;  as a consequence, this provided a new, easy, and direct proof of Morita's result that the this homomorphism lifted to the full mapping class group.  Following in this vein, \cite{bkp} used  groupoid versions of the Ptolemy and  Torelli groups of a bordered surface $\S1g$  in terms of $\pi_1$-marked and homology marked bordered fatgraphs to lift the higher Johnson homomorphisms to the groupoid level.  

In \cite{abp}, it was shown  that many other representations of the mapping class group $\M1g$ of a once bordered surface, i.e., a compact orientable surface with one boundary component, can be lifted to the groupoid level as well.  For example, the Nielsen embedding $N\colon \M1g\ra \Aut(F_{2g})$, thus also any representation which factors through it,  and the symplectic representation $\tau_0\colon \M1g\ra Sp(2g,\bZ)$ were lifted to groupoid representations $\tilde N$ and $\tilde \tau_0$ respectively.  These lifts depended on an algorithm, the greedy algorithm, which canonically constructed a maximal tree in a bordered fatgraph.  The lift $\tilde N$ had the nice property that  its basic form depended  only on six essential cases, two of which gave the identity in $\Aut(F_{2g})$; however, for some of these cases the explicit value of $\tilde N$ depended on the global combinatorics of the fatgraph in question.  

In describing these lifts, a connection between bordered fatgraphs and linear chord diagrams was introduced and exploited.  More specifically, the greedy algorithm provided a means to assign a linear chord diagram to each bordered fatgraph by ``collapsing''  the constructed maximal tree into a straight line.  Moreover, a correspondence between the elementary moves of fatgraphs, Whitehead moves, and the elementary moves of chord diagrams, chord slides, was established.    
  
  In this paper, we further explore the connection between bordered fatgraphs and chord diagrams.  In particular, we show that the groupoid generated by chord slides on certain linear chord diagrams gives another combinatorial version of the fundamental path groupoid of the \Teich space of a once bordered surface.  We moreover give a complete set of relations in this groupoid.   There seem to be several advantages of the chord diagrammatic viewpoint over the fatgraph one: first of all, the number of objects is considerably reduced, second of all, every elementary move has a simple yet nontrivial description in terms of the lift to $\Aut(F)$.  This may be advantageous in exploring lifts of the representation varieties.
  
 We introduce a dual version of chord diagrams which are more suited for a presentation of the Torelli groupoid.  As an application, we prove that the ``rational'' algorithm for the lift of the symplectic representation in \cite{abp} is in fact integral.
  
  To be as self-contained as possible, we shall repeat some of the results and terminology of \cite{abp}.

\section{Marked Bordered Fatgraphs}
\label{sect:markedborderedfatgraphs}

We begin by setting up some notation.
All graphs considered will be finite connected 1-dimensional CW-complexes.  We denote the set of oriented edges of a graph $G$ by $\mathcal{E}_{or}(G)$.
For an oriented edge $\mb{e}\in\mathcal{E}_{or}(G)$,  we let $\mb{\bar e}$ denote the same edge with the opposite orientation and let $v(\mb{e})$ denote the vertex that $\mb{e}$ points to.

A \emph{fatgraph} $G$ is a vertex-oriented graph, meaning there is   a  cyclic ordering assigned to the  oriented edges pointing to each vertex of $G$.  
 This additional  structure gives rise to certain cyclically ordered sequences of oriented edges called the \emph{boundary cycles} of $G$.  Specifically,  an oriented  edge $\mb{e}$ of a boundary cycle  is followed by the next edge in the cyclic ordering at $v(\mb{e})$, but with the orientation so  that it points away from $v(\mb{e})$.   In depicting a fatgraph, we will assume  the cyclic ordering at a vertex agrees with  the counterclockwise orientation of the plane so that  the boundary cycle of a depicted fatgraph $G$ can be thought of  as a path alongside  $G$ with  $G$  on the left.    
 
 We call any consecutive sequence of oriented edges in the boundary cycle of a fatgraph $G$ a \emph{sector}.  If $G$ is trivalent, there are three sectors naturally associated to each vertex of $G$, and there are four sectors associated to every edge of $G$.  We say that a fatgraph $G$ with $n$ boundary cycles has \emph{genus} $g$ if $\chi (G)=2-2g -n$.  For convenience, we shall consider only fatgraphs of a fixed genus $g$, unless otherwise stated.
 An isomorphism between two fatgraphs $G$ and $G'$ is a bijection of edges and vertices which preserves the incidence relations of edges with vertices and the cyclic ordering at each vertex.  We shall always regard isomorphic fatgraphs as equivalent.  

  We  now define our main combinatorial object:   A \emph{(once-)bordered fatgraph} is a fatgraph with only one boundary cycle such that all vertices are at least trivalent except for one which is univalent.   We call the edge incident to the univalent vertex the \emph{tail} and denote it by $t$.  Note that  an automorphism of a bordered fatgraph $G$  must map the tail to itself; thus, any automorphism of $G$ is necessarily trivial since any automorphism fixing an oriented edge of $G$  must also  fix all neighboring edges, thus the entire fatgraph.  

One consequence of the definition of a bordered fatgraph  is that it provides a natural linear ordering on its set of oriented edges.  This linear ordering $<$  can be defined by  setting $\mb{x}<\mb{y}$ if $\mb{x}$ appears before $\mb{y}$ while traversing the boundary cycle of $G$ beginning at the univalent vertex of the tail.    Since every  edge $e$ appears  in the boundary cycle once with each of its two orientations, we can define the \emph{preferred orientation} of $e$,  denoted simply by $\mb{e}\in\mathcal{E}_{or}(G)$, to be the one such that  $\mb{e}<\mb{\bar e}$.   Note that the tail with its preferred orientation is minimal in $\mathcal{E}_{or}(G)$, i.e., $\mb{t}\leq \mb{x}$ for all oriented edges $\mb{x}\in \mathcal{E}_{or}(G)$.

\subsection{Markings}

 Let $\S1g$ be a surface of genus $g$ with one boundary component and two fixed points $p\neq q\in\partial\S1g$ on the boundary.  We let $\pi_1$ denote the fundamental group of $\S1g$ with respect to the basepoint $p$,  and we let $H$ denote its abelianization.  
  The group  $\pi_1$ is non-canonically isomorphic to a free group $F_{2g}$ on $2g$ generators, and we shall often consider the boundary $\partial \S1g$ as a word in these generators under such an isomorphism.    By a classical result of Nielsen (see \cite{Zieschang}), the mapping class group $\M1g$  of $\S1g$  can be embedded  $N\colon \M1g \hra \Aut(F_{2g})$ in the automorphism group of $F_{2g}$ as the subgroup of all elements which preserve the word representing $\partial \S1g$.
 
 A \emph{marking} of a genus $g$ bordered fatgraph $G$ is an isotopy class of embeddings of $G$  into $\S1g$  as a spine so that the complement is contractible and the image of $v(\mb{\bar t})$ is the point $q$.    The mapping class group $\M1g$ acts freely on the set of all markings of a bordered fatgraph $G$.   Relying on the fact that dual to a marked bordered fatgraph is a triangulation of the surface $\S1g$ with all arcs based at the basepoint $p$, it was shown in \cite{bkp} that this notion of marking is equivalent to the following

\begin{definition}
A \emph{geometric $\pi_1$-marking} of a  bordered fatgraph $G$ is a map $\pi_1\colon \mathcal{E}_{or}(G)\ra \pi_1$ which satisfies the following compatibility conditions:
\begin{itemize}
\item {\bf (orientation)} For every oriented edge $\mb{e}\in \mathcal{E}_{or}(G)$,
\[
\pi_1(\mb{e}) \pi_1(\mb{\bar e})=1.
\]
\item {\bf (vertex)} For every vertex $v$ of $G$,
\[
\pi_1(\mb{e}_1) \pi_1(\mb{ e}_2) \cdots \pi_1(\mb{ e}_k)=1,
\]
where $\mb{e}_1, \ldots, \mb{e}_k$ are the cyclically ordered oriented edges incident on $v=v(\mb{e_i})$ for $i=1, \ldots, k$.
\item {\bf (surjectivity)} $\pi_1(\mathcal{E}_{or}(G))$ generates $\pi_1$.
\item {\bf (geometricity)} $
\pi_1(\mb{\bar t})=[\partial\S1g] $.
\end{itemize}
\end{definition}

Do to the equivalence of the two concepts, we shall not distinguish between markings and $\pi_1$-markings of a fatgraph.  We will also find it convenient to blur the distinction between  an oriented edge and its $\pi_1$-marking, and from now on we shall denote $\pi_1(\mb{e})$ simply by $\mb{e}$ and $(\pi_1(\mb{e}))^{-1}$  by $\mb{\bar e}$.

We now define the elementary move for trivalent fatgraphs.  
Given a trivalent bordered fatgraph $G$ and a non-tail edge $e$ of $G$, the \emph{Whitehead move} on $e$ is the  collapse of $e$ followed by the unique distinct expansion of the resulting four-valent vertex.  Note that any non-tail edge of $G$ necessarily has distinct endpoints since there is only one boundary cycle.

By the vertex compatibility conditions, it is easy to see that any $\pi_1$-marking  of  a bordered fatgraph $G$ evolves unambiguously under a Whitehead move.  Thus, there  is a natural partially defined composition on the set of Whitehead moves of marked bordered fatgraphs, where $W\colon G_0\ra G_1$ can be composed with $W'\colon G'_0\ra G'_1$ to obtain $W'\circ W$ if $G_1=G'_0$ as marked bordered fatgraphs.

\subsection{The Ptolemy groupoid}

It is well known that the  \Teich space of $\S1g$ admits an $\M1g$-equivariant ideal cell decomposition in terms of fatgraphs, the so called bordered fatgraph complex $\mathcal{G}_T$ of $\S1g$,  as developed  (first for punctured surfaces)  by Penner \cite{penner} in the hyperbolic setting and Harer-Strebel-Mumford \cite{harer} in the conformal setting.  See also \cite{Penner04} \cite{godin} for discussions of the bordered case.

For our purposes, it will be more convenient to work with the dual fatgraph complex $\dFGC$, which is an honest cell complex whose geometric realization $\grdFGC$ is homotopy equivalent to the \Teich space of $\S1g$.   The 2-skeleton of $\dFGC$ can be described quite concisely.  The  0-cells of $\dFGC$ correspond to trivalent marked bordered fatgraphs, and the oriented 1-cells correspond to the Whitehead moves between them.  The 2-cells correspond to fatgraphs which are the result of collapsing  two edges of a trivalent fatgraph  and come in two types depending on whether these two collapsed edges are adjacent or not.  

The mapping class group $\M1g$ acts freely on $\dFGC$.    As a result, any element of $\M1g$ can be represented by a path in $\mathcal{G}_T$ connecting 0-cells in the same $\M1g$ orbit, and these in turn can be represented by sequences of marked bordered fatgraphs beginning and ending at isomorphic  fatgraphs.   Moreover,  as  \Teich space   is known to be homeomorphic to a ball, each such path is unique up to homotopy;  thus, the element of $\M1g$ represented by a sequence of moves is trivial if and only if the sequence begins and ends with the same marked bordered fatgraph.

Recall that a groupoid is  a category whose morphisms are all isomorphisms.   By a \emph{full subgroupoid}, we shall simply mean a full subcategory of a groupoid.  Thus by definition, the subgroupoids of a groupoid $\mathfrak{G}$ are in one-to-one correspondence with the subsets of objects of $\mathfrak{G}$.  

A groupoid can also be viewed as a set of partially composable elements which satisfy conditions analogous to those which define a group.  In this way, it makes sense to speak of presenting a groupoid in terms of generators and relations.  In other words, a generating set of a groupoid is a set of morphisms  from which any morphism can be constructed by composition, and a set of relations is a set $\mathcal{R}$ of compositions of generators such that any  trivial morphism can be written as a composition of elements of $\mathcal{R}$.  

Our primary example of a groupoid is the fundamental path groupoid $\mathfrak{P}(X)$ of a topological space $X$.  The objects and morphisms of $\mathfrak{P}(X)$ are  in one-to-one correspondence with the points of $X$ and homotopy classes of paths between them, and we shall sometimes blur this distinction.    More precisely, we take the  objects of $\mathfrak{P}(X)$  to be the fundamental groups $\pi_1(X,p)$  based at points $p\in X$ and the morphisms to be the isomorphisms $\pi_1(X,p)\cong\pi_1(X,q)$ given by conjugation by a path connecting $p$ and $q$.   

 We now make a definition:
\begin{definition}
Let the Ptolemy groupoid $\Pt$ be the  full subgroupoid of the fundamental path groupoid $\tgrpd$ of $|\mathcal{\hat G}_T|$ whose objects (points) are  the images of 0-cells of $\dFGC$ (i.e., trivalent marked bordered fatgraphs) in $|\mathcal{\hat G}_T|$. 
\end{definition}

Note that $\Pt$ is an $\M1g$-equivariant subgroupoid of $\tgrpd$ in the sense that the full $\M1g$-orbit of any object  of $\Pt$ is contained in the objects of $\Pt$, and the action of $\M1g$ is natural with respect to  the morphisms of $\Pt$.  The fact that $\Pt$ is a full subcategory means that for every two objects of $\Pt$, the set of morphisms between them in $\Pt$ is  the full set of morphisms between them in the path groupoid of  $|\mathcal{\hat G}_T|$.  In this way, one is assured that all  information of the mapping class group is captured in $\Pt$ at every point. 

Moreover, by the properties of $\dFGC$, the morphisms of $\Pt$ are generated by Whitehead moves, and the relations between these generators are described by the 2-cells of $\dFGC$.  Thus the groupoid $\Pt$ has the advantage over the full path groupoid of $|\mathcal{\hat G}_T|$ in that it has a discrete set of objects and morphisms  with simple combinatorial descriptions.  

We now describe the relations in $\Pt$ and begin by introducing some notation.  Given any non-tail edge $e$ of $G$, we let $W_e\colon G\ra G'$ be the Whitehead move on $e$ and we let $e'$ be the the unique new edge of $G'$ produced by this move (while identifying all other edges in $G$ and $G'$).  For any edge $e$ of a bordered fatgraph $G$, let $I_e\in \Pt$ be the sequence of a Whitehead move on $e$ followed by its inverse:  $I_e=W_{e'}\circ W_{e}$.  The \emph{involutivity relation} states the obvious fact that this sequence of moves is a relation in $\Pt$.   Given any two edges $e$ and $f$ of $G$ which are not adjacent, meaning they are not both incident to a common vertex,  we define the \emph{commutativity sequence} on $e$ and $f$ by  $C_{e,f}=W_{f'}\circ W_{e'}\circ W_{f}\circ W_{e}$.  Finally, if $e$ and $f$ are adjacent, we define the \emph{pentagon sequence} by $P_{e,f}=W_{e''}\circ W_{f'}\circ W_{e'}\circ W_{f}\circ W_{e}$.  See figure \ref{fig:14325}  One can check directly that under these sequences of moves, the initial and final marked bordered fatgraphs are isomorphic, therefore both $C_{e,f}$ and $P_{e,f}$ are trivial in $\Pt$.  We call   these the commutativity and pentagon relations respectively.

\section{The Greedy Algorithm}

In \cite{abp}, an algorithm was introduced to produce a maximal tree $T_G$ in a given bordered fatgraph $G$.  This algorithm, called the \emph{greedy algorithm}, works by building the tree $T_G$ one edge at a time while  traversing the boundary cycle starting at the tail.   Every edge is added to  $T_G$ as long as the resulting subgraph  is still a tree.  The advantage of this algorithm is that it is completely local in the sense that one can determine if an edge $e$  of $G$ is contained in $T_G$ by simply looking at the order in which the sectors associated  to $v(\mb{e})$ are traversed.  Put concisely, $e\in T_G$ if and only if $\mb{e}\leq \mb{x}$ for all  $\mb{x}\in\mathcal{E}_{or}(G)$ with $v(\mb{x})=v(\mb{e})$.  

We denote by $X_G$ the complement of $T_G$ in $G$.  The edges in this set are linearly ordered by $<$, and we let $\mb{X}_G$ denote the  linearly ordered subset of $\mathcal{E}_{or}(G)$ obtained by assigning to each edge in $X_G$ its preferred orientation.   We call $\mb{X}_G$ the set of generators of $G$.   The existence of $T_G$, or more to the point  $\mb{X}_G$, has immediate consequences as follows.  First of all, given a $\pi_1$-marking of $G$, the image of $\mb{X}_G$ under the map $\pi_1$ produces a set of generators, which we also denote by $\mb{X}_G$, for the free group $\pi_1$, thus also an explicit isomorphism $\pi_1\cong F_{2g}$.  Moreover, under a Whitehead move $W\colon G\ra G'$, the two isomorphisms provided by $G$ and $G'$ can be composed to give an isomorphism from $F_{2g}$ to $F_{2g}$,  i.e.\  an element $\tilde{N}(W)\in \Aut(F_{2g})$.  This map $\tilde N \colon \Pt \ra \Aut(F_{2g})$  is the groupoid lift of Nielsen's embedding as mentioned in the introduction (see \cite{abp}).

It was shown that these elements $\tilde{N}(W)\in \Aut(F_{2g})$ are computable from the initial fatgraph $G$ of the Whitehead move, and moreover, that they fall into 6 essential cases depending on the order of traversal of the sectors associated to the edge of the Whitehead move.  We depict these cases in Figure  \ref{fig:WCases} where we have numbered the sectors associated to a Whitehead move $G\ra G'$ according to their order of traversal, and we have used check marks to denote edges that must be generators and  question marks to denote edges that may be generators.    We say that  a Whitehead move $W$ is a \emph{type $k$} move if it or its inverse corresponds the the $k$th case  according to our labeling in this figure.

\begin{figure}[!h]
\begin{center}
\epsffile{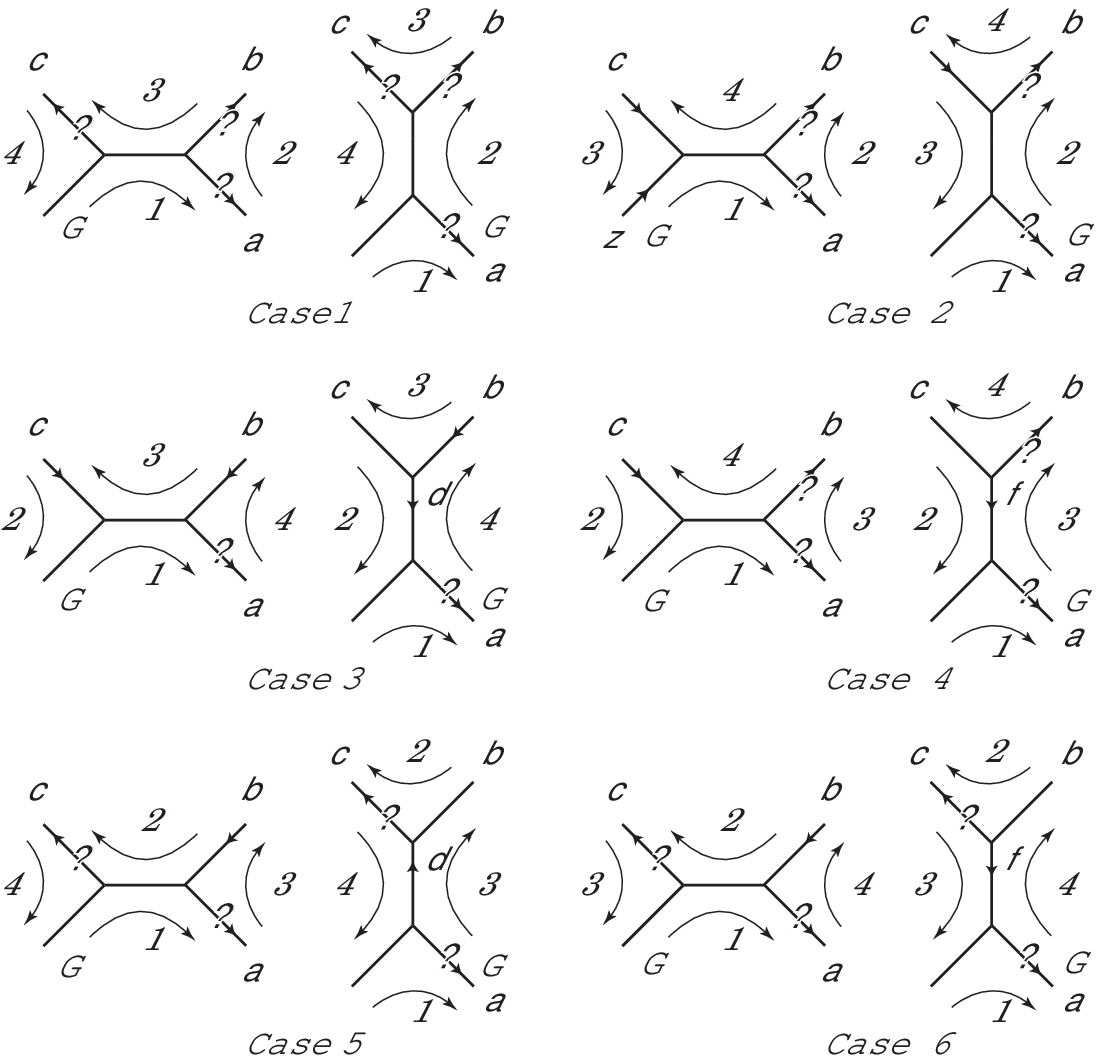}
\caption{Six cases of Whitehead moves $W\colon G\ra G'$}
\label{fig:WCases}
\end{center}
\end{figure}

\subsection{Essential cases of $\widetilde{N}(W)$}

We quickly recap the analysis of the six cases in \cite{abp}.   For moves of type 1 and type 2, it is easy to see that the corresponding elements of $\Aut(F_{2g})$ are the identity.  In fact, in \cite{abp} it was shown that the kernel of $\tilde{N}$ was generated by these two types of moves.

For a type 3 Whitehead move as in the figure, a  generator $\mb{c}$ gets mapped to a new generator $\mb{d}$ while all other generators are kept fixed.  By the vertex compatibility condition for $G'$, we have  $\mb{b}\mb{c}\mb{\bar d}=1$ so that  $\mb{c}\mapsto \mb{d}=\mb{ b}\mb{c}$.  

For case 4, the situation is similar to case 3 except that $\mb{b}$ need not be a generator, and we have the slightly different relation $\mb{c}\mapsto \mb{f}=\mb{\bar b}\mb{c}$.  

For a type 5 move, one can again use the vertex compatibility relation to find an expression for the new generator.  For the Whitehead move in the figure, the generator $\mb{b}$ is removed from the generating set, while  the new generator is given by  $\mb{c}\mb{\bar b}$.  However, in this case, it is not necessarily true that the element of $\Aut(F_{2g})$ is given by $\mb{b}\mapsto\mb{c}\mb{\bar b}$.  In particular,  if there are generators $\{\mb{x_i}\}_{i=\ell}^m\subset \mb{X}_G$ with $\mb{b}<\mb{x}_i<\mb{\bar b}$, then the ordering of the  generators $\mb{x}_\ell,\ldots,\mb{x}_m,\mb{c}\mb{\bar b}\in\mb{X}_{G'}$   will be cyclically permuted as well. 
 Also, as in case 4, one may need to write $\mb{c}$ in terms of the generators  of $G$ in order to write an explicit element of $\Aut(F_{2g})$.

For the Whitehead move of type 6  we have a situation which is essentially identical to that of case 5 except that now the new generator $\mb{f}$ has an orientation which is opposite that of the generator $\mb{d}$ of case 5. 

\section{Chord Diagrams}

We define a \emph{linear chord diagram} $C$ with $k$ chords to be a graph immersed in the plane with only double point singularities, such that the image consists of the   interval $[0,2k]$, called the \emph{core} of $C$,  together with $k$ immersed line segments,  called the \emph{chords} of $C$, whose endpoints are attached to the core  at the integer points $\{1, \ldots, 2k\}$.  We shall often identify a chord diagram with its image.  From now on, all linear chord diagrams will simply be called chord diagrams.   If the chords of a diagram are oriented, we call it an \emph{oriented chord diagram}. 
 
  Due to the orientation of the plane, the underlying graph of every chord diagram $C$ can be endowed with a fatgraph structure,  and we denote this fatgraph by $G_C$.    Note that by our definition above, all chord diagrams come with a tail given by the interval $[0,1]$, which extends to the left of the diagram, and that all  vertices other than the one at  $(0,0)$ are trivalent (so that we do not consider the double points as vertices).  Also note that the rightmost point of the core of a diagram is not a vertex of $G_C$; however, no harm is done if one  views it as a bivalent vertex, and we shall at times take this point of view.   We say that a linear chord diagram $C$ is a \emph{bordered chord diagram} if  the corresponding fatgraph $G_C$  is a (once-) bordered fatgraph, meaning it has only one boundary component.   We also define a marking of $C$  to simply be a marking of $C_G$. 
 
\begin{definition}
We define the marked chord groupoid $\Cs$  to be the full subgroupoid  of $\Pt$ whose  objects correspond to marked bordered  chord diagrams.
\end{definition}

Again, since $\Cs$ is a full $\M1g$-invariant  subcategory, it will contain all the relevant information of the mapping class group.  One of the main goals of this note is to give a concise combinatorial presentation of this groupoid, which essentially means giving a complete and efficient set of  generators and relations.  We begin with some notation.  
 
As with edges of a fatgraph, we shall let $\mb{c}$ and $\mb{\bar c}$ denote the two orientations of a chord $c$.  Additionally, we shall slightly abuse notation and  let $\mb{c}$ also denote the endpoint of $c$ that the oriented chord $\mb{c}$ points away from.  See figure \ref{fig:CS}.    
 Note that under  this notation, $\mb{c}$ will label the end  of $c$ which is incident  to $v(\mb{\bar c})$.  While slightly confusing, this notation will be convenient for the proof of Lemma 5.1.
 
  The linear left-to-right  ordering of points on the core gives rise to a linear ordering on the ends of the chords which we denote by $\prec$.  We also will write $\mb{c}\iprec\mb{d}$ if $\mb{c}$ immediately precedes $\mb{d}$ with respect to this ordering

\subsection{Chord slides}
We now define the elementary move for chord diagrams, the  \emph{chord slide}.   Assume that  $\mb{c}\iprec\mb{d}$. The slide of $\mb{c}$ along $\mb{d}$ moves the endpoint $\mb{c}$ of $c$ along the chord $d$ from $\mb{d}$ to its other end $\mb{\bar d}$ while keeping all other chord endpoints fixed.  If we denote by $\mb{c}'$ and $\mb{d}'$ the corresponding endpoints of chords of the new diagram, then we have  $\mb{\bar d'}\iprec \mb{c'}$.  See figure \ref{fig:CS}.  Similarly, the slide of $\mb{d}$ along $\mb{c}$ moves $\mb{d}$ so that $\mb{d'}\iprec\mb{\bar c'}$.  

\begin{figure}[!h]
\begin{center}
\epsffile{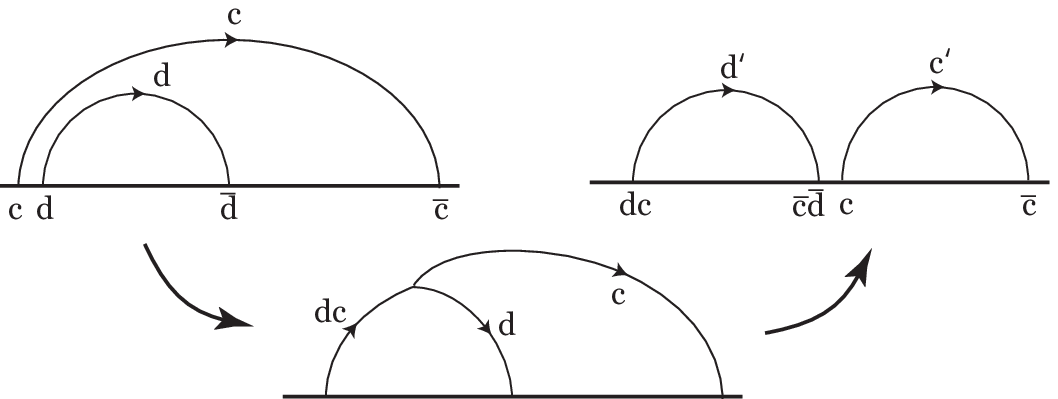}
\caption{Chord slide of $\mb{c}$ along $\mb{d}$.}
\label{fig:CS}
\end{center}
\end{figure}

\begin{observation}\label{obs:WCS}
Every chord slide of $C$ can be realized as a composition of two Whitehead moves on $G_C$, exactly one of which is of type 1 or 2. 
\end{observation}
For example, if $\mb{c}\iprec\mb{d}$,   then  the slide of $\mb{c}$ along $\mb{d}$ is given by the Whitehead move on the edge of the core lying between $\mb{c}$ and $\mb{d}$  followed by the Whitehead move on $d$. See figure \ref{fig:CS}. Similarly, the slide of $\mb{d}$ along $\mb{c}$ is the composition of the Whitehead move on the edge of the core between them  followed by the Whitehead move on $c$.  The fact that one of the two Whitehead moves realizing a chord slide is of type 1 or 2 is easily verified by considering all possible cases  of orderings of  the sectors associated to the two chords of the chord slide.  

As a consequence of the above observation,  any marking of a chord diagram evolves unambiguously under a chord slide.  For example, in figure \ref{fig:CS}, we have the marking of $\mb{c}$ evolve as $\mb{c}\mapsto \mb{c'}=\mb{c}$ while the marking of $\mb{d}$ evolves as $\mb{d}\mapsto \mb{d'}=\mb{dc}$. 

Another consequence of Observation \ref{obs:WCS} is that  the set of chord slides on marked bordered chord diagrams is a subset of the morphisms of $\Cs$.  We will show in Theorem \ref{thm:generate} that this is in fact a generating set, meaning that for any two marked bordered chord diagrams, one can be taken to the other by a sequence of chord slides.

To get a feel for the evolution of markings under chord slides, we now consider the six cases, or types, of chord slides, which are analogous to the six types of Whitehead moves.  We depict these cases in figure \ref{fig:CsC}.  In the figures, we have used thick lines to denote segments of the core, while we have used thin lines to denote chords.  Also,  we have labeled the sectors of the chord diagram according to their order of  traversal along the boundary component, and we have oriented each chord according to its preferred orientation.
  
\begin{figure}[!h]
\begin{center}
\epsffile{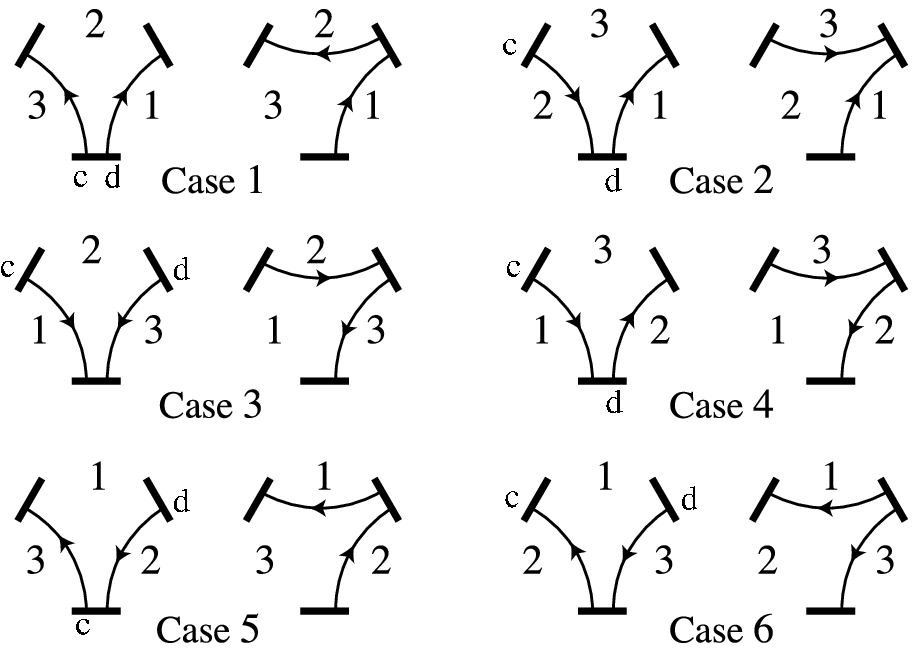}
\caption{Six cases of chord slides $W\colon G\ra G'$}
\label{fig:CsC}
\end{center}
\end{figure}
  
In each of the six cases, the chord labeled $c$ has one of its ends slid along one of the ends of the chord $d$.  The resulting effect for the set $\mb{X}_{G_C}$ is that the generator $\mb{d}$ is removed and a new generator is added in its place.   We find that the new generator is   $\mb{c}\mb{d}$,  $\mb{\bar c}\mb{d}$,   $\mb{d}\mb{c}$,  $\mb{\bar d}\mb{c}$,   $\mb{c}\mb{\bar d}$, or $\mb{d}\mb{\bar c}$ for  cases 1 through 6 respectively.   Also, there is the possibility that  the ordering of a subsequence of  generators of $\mb{X}_{G_C}$ are cyclically permuted in cases 3 through 6.  
 Note that all these expressions are relatively simple, completely explicit, and local, unlike the expressions for the six cases of Whitehead moves. 
    
\section{Fatgraph -- Chord Diagram Correspondence}  
 
 In the previous section, we showed that every linear chord diagram $C$ can be considered as a fatgraph, which we denoted by $G_C$.  For the remainder of this paper, we shall slightly abuse notation and simply denote the corresponding fatgraph by $C$.  
Conversely, in \cite{abp}, it was shown that for each bordered fatgraph $G$, there exists a sequence of moves of type 1 and 2 which takes $G$ to a new fatgraph, which we denote by $C_G$,  having  the form of a chord diagram.  We call this procedure  the \emph{branch reduction} algorithm and summarize some of its properties here.  

Given $G$, let $T_G$ be the maximal tree obtained via the greedy algorithm and let $S_G\subset T_G$ consist of those edges of $T_G$ that come before any generators with respect to the ordering $<$.  Note that $S_G$ must be connected.  By straightening $S_G$ to a horizontal  line, we can view $T_G$ as as a sequence of subtrees $T_i$  ``planted'' into $S_G$.  
  We call the edge of $T_i$ incident  to $S_G$ the \emph{trunk} of $T_i$.  It is easy to see that the Whitehead move on any  trunk is of type 1 or 2, so that it does not change the set of generators $\mb{X}_G$.  Moreover, under such a move, the length of $S_G$ is increased, so that by repeated applications, we arrive at a graph $C_G$ with $S_{C_G}=T_{C_G}$ and $\mb{X}_{C_G}=\mb{X}_G$.  This graph can be considered  an oriented chord diagram with oriented chords labeled  by $\mb{X}_G$ in a unique way.  In particular, we can unambiguously associate an edge of $G$ to every chord of $C_G$.   
   
\begin{figure}[!h]
\begin{center}
\epsffile{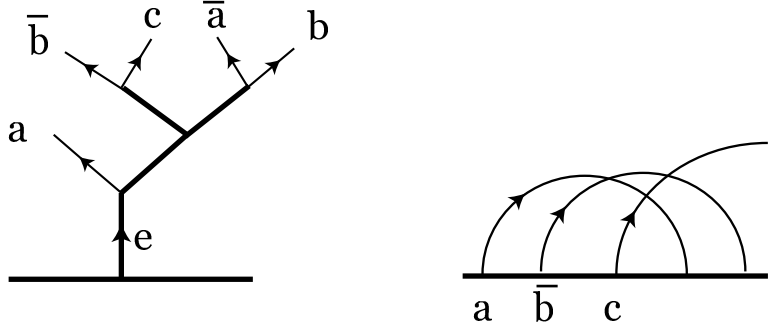}
\caption{Illustration of branch reduction.}
\label{fig:brred}
\end{center}
\end{figure}
  
  If $e$ is a trunk in $T_G$ for some $G$, then we denote by $T(\mb{e})$ the tree of which it is a trunk.  More generally, given any oriented edge $\mb{e}$ of $T_G$, we denote by $T(\mb{e})$ the subtree of $T_G$ rooted by $\mb{e}$ so that $\mb{e}$ points towards $T(\mb{e})$.  Note that either $T(\mb{e})$ or $T(\mb{\bar e})$ contains the tail of $G$.  Assume that $T(\mb{e})$ does not contain the tail.  Then  each univalent vertex  of $T(\mb{e})$ other than $v(\mb{\bar e})$ is adjacent to  two generators of $G$, and each bivalent vertex of $T(\mb{e})$ is adjacent to the end of one generator.  We call  these generators the \emph{leaves} of $T(\mb{e})$ and give them the orientation so that they point away from $T(\mb{e})$.  Note that there is a natural clockwise ordering to the leaves of  $T(\mb{e})$, and if $G$ is marked, then by repeated application of the vertex compatibility condition, the marking of $\mb{e}$ can be immediately read off as the product of these leaves in this ordering.  For example, the marking of the edge $\mb{e}$ in figure \ref{fig:brred} is given by $\mb{e}=\mb{a \bar b c \bar a b}$.
  Moreover,  under the branch reduction algorithm, the ordered leaves of $T(\mb{e})$  correspond to a sub-sequence of chord ends of $C_G$ in such a way that the ordering of the leaves coincides with the ordering given by $\prec$.
  
  Conversely,  assume $\mb{x}_1\iprec\mb{x}_2\iprec \dots \iprec\mb{x}_k$ are ends of chords such that $\mb{x}_k<\mb{x}_i$ for $1\leq i<k$.    Since $\mb{x}_k<\mb{x}_i$, a procedure opposite to the branch reduction algorithm may be applied to ``grow'' a planted tree $T$  from the core of $C$ using moves of  type 1 and 2 such that the  leaves of $T$ are $\{\mb{x}_1, \ldots, \mb{x_k}\}$ with this clockwise order.   
   
\subsection{The Whitehead move -- chord slide correspondence} \label{sect:WCScorrespondence}
 As we have seen, the branch reduction algorithm  provides  a map from marked bordered fatgraphs to marked chord diagrams.     Since both $\Pt$ and $\Cs$ are by definition full subcategories of the fundamental path groupoid of $\dFGC $, this defines a full and faithful functor from $\Pt$ to $\Cs$ which is the inverse of the inclusion $\Cs  \hra \Pt$.  Our next goal is to concisely describe the images of the morphisms of $\Pt$ under this functor.  We shall need the following lemma.
      
\begin{figure}[!h]
\begin{center}
\epsffile{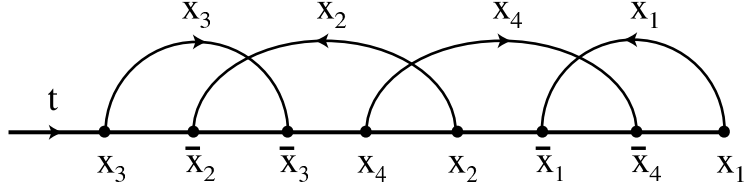}
\caption{Illustration of observation in proof of Lemma \ref{lem:unique}.}
\label{fig:obs}
\end{center}
\end{figure}

\begin{lemma}\label{lem:unique}
Given a set of generators $\mb{X}$ of $\pi$, there is at most one chord diagram $C$ whose set of generators $\mb{X}_C$ is equal to  $\mb{X}$ up to  permutation and replacement of any number of  elements $\mb{x}_i\in \mb{X}$ by their inverses $\mb{\bar x}_i$.
\end{lemma}
\begin{proof}
This lemma follows from some simple observations.  Let  $C$ be such a diagram.  By repeated application of the orientation and vertex compatibility conditions, the word representing $\mb{t}$ in the letters $\mb{X}_C$  can be directly computed  from the chord diagram $C$.  
Namely, by labeling the chord endpoints along the core of $C_G$,  $\mb{t}$ is obtained by simply multiplying these elements in their left-to-right ordering.    For example, in Figure \ref{fig:obs}, we have that $\mb{t}=\mb{x}_3\mb{\bar x}_2\mb{\bar x}_3\mb{x}_4\mb{x}_2\mb{\bar x}_1\mb{\bar x}_4\mb{x}_1$.  Moreover, the word representing $\mb{t}$ obtained in this way is reduced since the fatgraph $C$ has only one boundary cycle.
Since there is a unique reduced word representing any element of $\pi_1$ in a given set of letters, the lemma follows.
\end{proof}

As an immediate  consequence we have the following
\begin{corollary} \label{cor:settellsall}
The marked bordered chord diagram $C_G$ arising from a marked bordered fatgraph $G$ is uniquely determined by the (not necessarily ordered) set $\mb{X}_G$.
\end{corollary}

  We are now ready to state the following
  \begin{theorem}\label{thm:generate}
  The morphisms of $\Cs$ are generated by chord slides.
  \end{theorem}
Since $\Cs$ is a full subcategory of $\Pt$ and since  Whitehead moves generate   the morphisms of $\Pt$, this theorem follows immediately from the following lemma.    
  \begin{lemma} \label{lem:WmoveisCslide}
  For a Whitehead move $W_e\colon G\ra G'$ on an edge $e$ of $G$, let $\mathfrak{Cs}(W_e)\colon C_G\ra C_{G'}$  denote  the corresponding morphism in $\Cs$.  Then $\mathfrak{Cs}(W_e)$  is represented by  a sequence of chord slides along a single chord $\mb{x}$.  Moreover, considering  $x$ as an edge of $G$, $x$ is adjacent to $e$ in $G$.
  \end{lemma}
  \begin{proof}
  Consider the Whitehead move $W_e$ on an edge $e$.  
  Without loss of generality, we assume that the first (of the four) sectors associated to the edge $e$ traversed by the boundary cycle includes the edge $e$ itself,  as is the case for the fatgraphs labeled $G$ in  Figure \ref{fig:WCases}.  If $W_e$ is of type 1 or 2, then there is no change to $\mb{X}_G$, so the chord diagrams $C_{G}$ and $C_{G'}$ coincide; i.e., they are related by zero chord slides. 
  
   Now consider   the type 3 move $W_e$ as labeled in Figure \ref{fig:WCases}.   By applying the branch reduction algorithm to $G$,  we have the relation $\mb{\bar c}\iprec \mb{\bar b}$  in $C_G$.  Under the Whitehead move $W_e$, we have that $\mb{X}_{G'}$ is obtained from $\mb{X}_G$ by replacing $\mb{c}$ with $\mb{bc}$.  On the other hand, if one performs the chord slide of $\mb{\bar b}$ along $\mb{\bar c}$ (this is the inverse of a type 2 chord slide as in figure \ref{fig:CsC}), one finds that the generating set of $\mb{X}_{C_G}=\mb{X}_g$ is altered by replacing $\mb{c}$ with $\mb{bc}$.  By Corollary  \ref{cor:settellsall}, we see that  the chord diagram obtained by this slide is precisely equal to $C_{G'}$.   Since there is a unique morphism in $\Pt$ connecting any two objects, the morphism of $W_e\in\Pt$ maps to the above described chord slide in $\Cs$.
  
  Now consider a type 4 move again labeled by Figure \ref{fig:WCases} where $\mb{b}$ may or may not be a generator.  One can check that the tree $T(\mb{ b})$ does not contain the tail (this is true for every oriented edge with a question mark in Figure \ref{fig:WCases}).   Let $\{\mb{b}_i\}_{i=1}^{m}$ be the set of leaves of the tree $T(\mb{ b})$, linearly ordered by $\prec$ so that in $C_G$ $\mb{\bar c}\iprec \mb{b}_1\iprec \dotsm \iprec \mb{b}_m$ (note that if $\mb{b}$ is a generator, then $T(\mb{b})=b$ so that  $m=1$ with $\mb{b_1}=\mb{b}$).  Sliding each $\mb{b}_i$ along $\mb{\bar c}$, one obtains a new chord diagram with generating set differing from $\mb{X}_{C_G}$ only in that $\mb{c}$ has been replaced by $\mb{\bar b}\mb{c}$ with $\mb{b}=\prod \mb{ b}_i$.  Thus, the effects of these chord slides match the effects of the Whitehead move $W_e$, so that the image of $W_e$ in $\Cs$ is given by a sequence of chord slides along $\mb{\bar c}$.  
  
  Completely analogous arguments show that the type 5 and 6 moves similarly map to sequences of chord slides over the single chord $\mb{\bar b}$ of $C_G$.  
  
The last statement follows since the edge $e$ is adjacent to the edges  $b$ and $c$.
\end{proof}
  
Note that this lemma implies that we have more than just  a functor $\mathfrak{Cs}\colon \Pt \ra \Cs$.  We  actually have defined a map   which takes a word in the generators of $\Pt$ to an explicit word in the generators of $\Cs$.  Moreover, this map  takes a composition of two Whitehead moves representing a chord slide to the same chord slide.  

We now address the question of what sequence of chord slides in a chord diagram $C$ can arise as $\mathfrak{Cs}(W)$ for a  single Whitehead move $W$ on some bordered fatgraph $G$ with $C_G=C$.  As we have just seen, all slides of $\mathfrak{Cs}(W)$ must be along a single chord $\mb{x}_0\in \mb{X}_G$.  Also,  it is clear that any chord slide of a single chord can be obtained by a single Whitehead move (see Observation \ref{obs:WCS}).  
  
  \begin{lemma}\label{lem:whichslides}
  If $\{\mb{x}_0, \mb{x}_1, \ldots, \mb{x_k}\}$ with $k>1$ are chord endpoints for some bordered chord diagram $C$  such that $\mb{x}_0\iprec\mb{x}_1\iprec \dots \iprec\mb{x}_k$, then there is a Whitehead move $W$ on some bordered fatgraph $G$ with $C_G=C$  such that $\mathfrak{Cs}(W)$ is the slide of $\{\mb{x}_1, \ldots, \mb{x_k}\}$  over $\mb{x}_0$ if and only if $\mb{\bar x}_0<\mb{x}_k<\mb{x}_i$ for all $0\leq i<k$.  Similarly, if $\mb{x}_1\iprec \dots \iprec\mb{x}_k\iprec\mb{x}_0$, then there is a Whitehead move $W$ on some $G$  with $\mathfrak{Cs}(W)$ the slide of  $\{\mb{x}_1, \ldots, \mb{x_k}\}$ over $\mb{x}_0$ if and only if $\mb{x_k}<\mb{x}_i$ for all $0\leq i<k$.
  \end{lemma}
  \begin{proof}For the first part, 
  assume  that $\mb{x}_0\iprec\mb{x}_1\iprec \dots \iprec\mb{x}_k$ and $\mb{\bar x}_0<\mb{x}_k<\mb{x}_i$ for $0\leq i<k$.    Since $\mb{x}_k<\mb{x}_i$ for $0\leq i<k$, we may grow  a planted tree $T$  from the core of $C$  such that the  leaves of $T$ are $\{\mb{x}_1, \ldots, \mb{x_k}\}$ with this clockwise order (see the  paragraph preceding  Section \ref{sect:WCScorrespondence}).   Since $\mb{\bar x}_0<\mb{x}_k<\mb{x}_i$ for $0\leq i<k$, the Whitehead move on the edge between $\mb{x}_0$ and the trunk of $T$ is of type 4, so that $T=T(\mb{b})$ and $\mb{x}_0=\mb{\bar c}$ in the notation of figure \ref{fig:WCases}.   One can verify that this Whitehead move corresponds to the desired slide.   
 
 Conversely, in order for  the chord endpoints $\{\mb{x}_1, \ldots, \mb{x_k}\}$ to all be slid along $\mb{x}_0$ under a single Whitehead move on a bordered fatgraph $G$, this move must be of type 4.   Again using the notation of figure \ref{fig:WCases},   the chord  endpoints $\{\mb{x}_1, \ldots, \mb{x_k}\}$ correspond to the leaves of the subtree $T(\mb{b})$, and  the $k+2$ sectors of $T(\mb{b})$ correspond to the sectors of $C$ which are to the right or left of one of the $\mb{x}_i$. From figure \ref{fig:WCases}, we see that the sector to the left of $\mb{x}_0$ must be the first to be traversed, and the sector to the right of $\mb{x}_k$ must be the next.  In terms of the ordering $<$, this translates to  $\mb{\bar{x}}_0 < \mb{x}_i$ for all $0\leq i \leq k$ and $\mb{x}_k< \mb{ x}_i$ for all $0\leq i<k$. 
 
  For the second part, a similar analysis shows that for a set of generators with $\mb{x}_1\iprec \dots \iprec\mb{x}_k\iprec\mb{x}_0$ and $\mb{x_k}<\mb{x}_i$ for all $0\leq i<k$, a planted tree can be built with leaves $\{\mb{x}_1, \ldots, \mb{x_k}\}$.  This proves the forward implication as above.  Conversely, a slide of $\{\mb{x}_1, \ldots, \mb{x_k}\}$ along $\mb{x}_0$ must arise from a move of type 5 or 6, in which case figure \ref{fig:WCases} shows that  the first sector to be traversed is the one to the left of $\mb{x}_0$.  Thus,  $\mb{\bar x}_0<\mb{x_k}<\mb{x}_i$ for all $0\leq i<k$, as desired.  
  \end{proof}
 
\section{Relations  in $\Cs$}

In the previous section, we proved that the groupoid $\Cs$ is generated by chord slides  and gave an explicit map $\mathfrak{Cs}\colon \Pt\ra\Cs$ which took a composition of generators of $\Pt$ to a  composition of generators of $\Cs$.   In this section, we give an explicit description of the relations satisfied by these generators.  

We begin by diagrammatically listing in Figure \ref{fig:relations}  the following relations:
\begin{figure}[!h]
\begin{center}
\epsffile{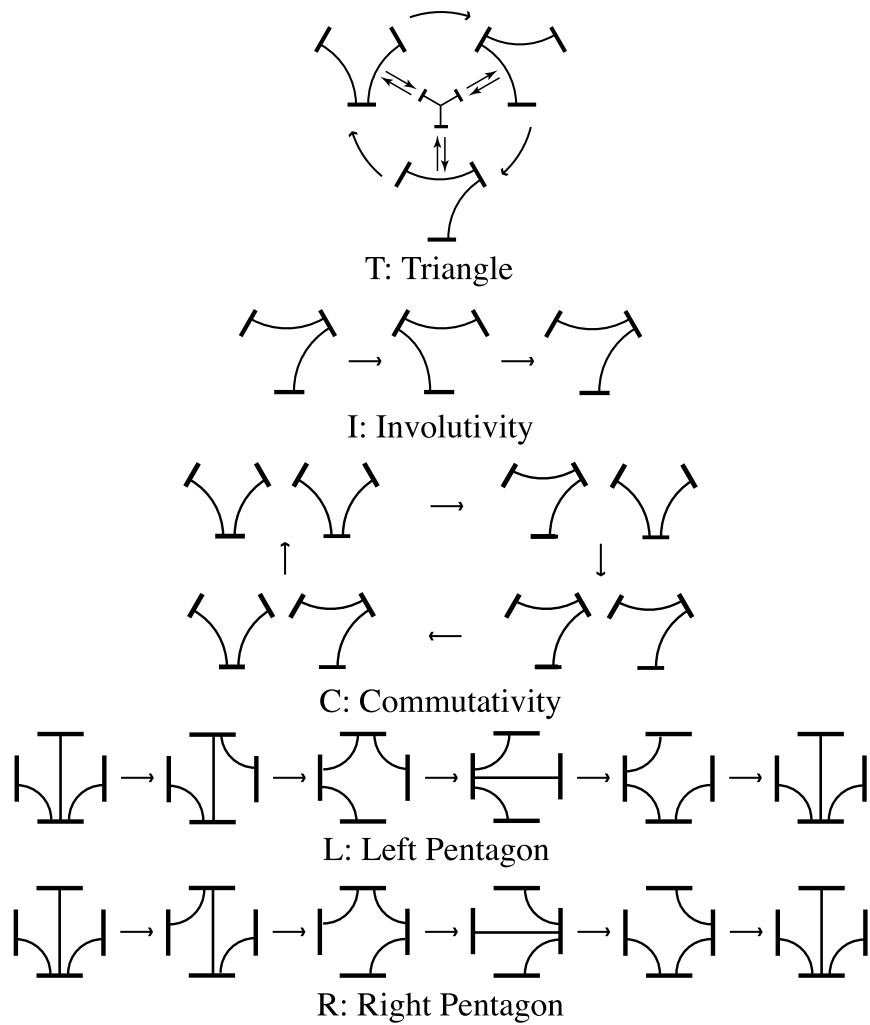}
\caption{Main relations}
\label{fig:relations}
\end{center}
\end{figure}

\begin{itemize}
\item {Triangle $T$:} This is perhaps the most surprising relation which says that three consecutive slides of two chords along each other is an identity.  
\item Involutivity $I$: This relation states the obvious fact that a chord slide followed by its  inverse  is the identity in $\Cs$.
\item Commutativity $C$:  This relation says that chord slides of two distinct ends along disjoint chords commute with each other.   
\item Left Pentagon $L$: This relation says that two ways to slide two adjacent chord endpoints to the left are equivalent.  
\item Right Pentagon $R$:  This relation says that two ways to slide two adjacent chord endpoints to the right are equivalent. 
\end{itemize}

Viewed as a sequence of 6 Whitehead moves, the relation $T$ follows from involutivity in $\Pt$, as illustrated in Figure \ref{fig:relations}.  The  relations $I$, $C$,  $L$, and $R$ can be verified similarly. 

\begin{figure}[!h]
\begin{center}
\epsffile{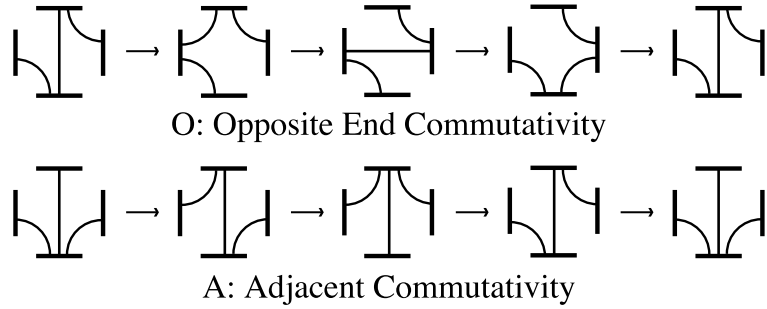} 
\caption{The $O$ and $A$ relations.}
\label{fig:oacomm}
\end{center}
\end{figure}

It will be helpful to also introduce the two relations  shown in figure \ref{fig:Ocomm}, which are similarly verified.
\begin{itemize}
\item Opposite End Commutativity $O$:  This relations says that the slides of opposite ends of the same chord commute with each other.

\item Adjacent Commutativity $A$:  This relation says that chord slides of two ends along the same chord (when possible) commute. 
\end{itemize}

\begin{lemma}
The relations $O$ and $A$ follow from the relations $T$, $L$, and $R$.  
\end{lemma}
\begin{proof}
The proofs of the two statements are similar and are 
essentially contained in figure \ref{fig:Ocomm}, where the outside ring of each diagram is an $R$ relation.  The inside of the first  is  decomposed  into an $L$ relation, two $T$ relations, and an $O$ relation. The second is decomposed  into an $L$ relation, two $T$ relations, and an $A$ relation. 
\end{proof}

\begin{figure}[!h]
\begin{center}
\epsffile{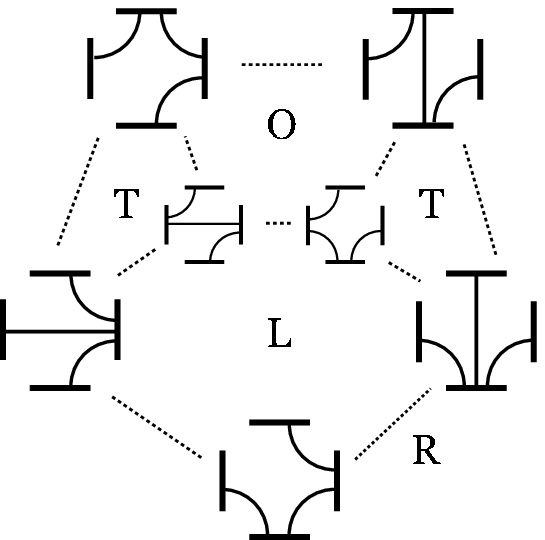} \quad
\epsffile{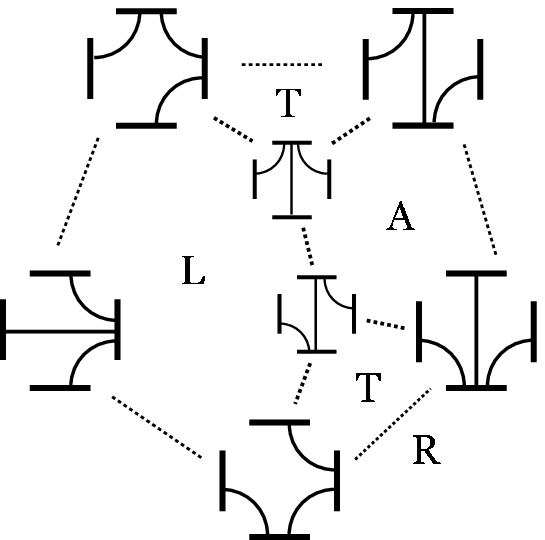}
\caption{Derivation of the $O$ relation.}
\label{fig:Ocomm}
\end{center}
\end{figure}

\begin{figure}[!h]
\begin{center}
\epsffile{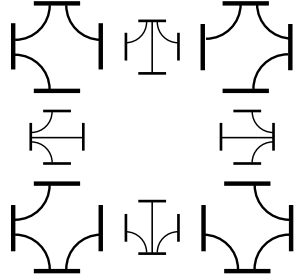}
\caption{The square relation.}
\label{fig:square}
\end{center}
\end{figure}

All of the above relations, with the exception of $T$,  also have versions where one or two  of the chords are replaced by several adjacent  chords which are simultaneously slid.   These multiple-chord slide relations all follow from the above by an easy application of induction.  We leave the details to the reader.  One can check directly that the multiple chord versions of $T$  no longer hold as relations.  However, while not needed for the sequel, we do mention that there are generalizations of the triangle relation $T$ which one might call the square, pentagon (not related to pentagon in $\Pt$), hexagon, etc.\ relations.  We depict the square relation in figure \ref{fig:square}.  These can be derived from the  relations $T$, $I$, $C$, $L$, and $R$, either directly as in the proof of the previous lemma, or by making use of the following theorem.   

\subsection{Main theorem}
\begin{theorem}
Any relation in $\Cs$ can be derived from the relations $T$, $I$, $C$,  $L$, and $R$.
\end{theorem}
\begin{proof}
To prove this, we rely on the fact that the relations in $\Pt$ are generated by the involutivity, commutativity, and pentagon relations.  In particular, any sequence of chord slides representing a trivial morphism in $\Cs$  must be a product of  images of these relations under the map $\mathfrak{Cs}$.  
We shall analyze each of these relations in turn and show that each one translates to a composition of the above relations $T$, $I$, $C$, $L$, $R$, $O$, and $A$ in $\Cs$.

First, it is clear that  the involutivity relation $I$ in $\Cs$  implies involutivity in $\Pt$.  Indeed, given a Whitehead move $W_e$ in $\Pt$, we have by Lemma \ref{lem:WmoveisCslide} that the corresponding element of $\Cs$ is a slide of some number, say $n$, adjacent chord endpoints along a single chord.  The inverse  Whitehead move $W_{e'}$ is similarly a slide of the same chord endpoints back along the  same chord to their original positions.   Thus, in this case the relation  $W_e\circ W_{e'}$ translates into the $n$-chord version of  relation $I$.  As mentioned above, the multiple-chord version of $I$ follows from $I$ by a simple induction.

Now consider the commutativity relation in $\Pt$.  Let $W_e$ and $W_f$ be the two Whitehead moves of such a relation so that the edges $e$ and $f$ are not adjacent.  Let $\{\mb{y}_1, \ldots, \mb{y_k}\}$ be the chord endpoints which are slid under the move $\mathfrak{Cs}(W_e)$, and let $\mb{y}_0$ be the endpoint of the chord which the $\{\mb{y}_1, \ldots, \mb{y_k}\}$ are slid along.  Similarly, let  $\{\mb{z}_1, \ldots, \mb{z_m}\}$ be the chord endpoints which are slid by   $\mathfrak{Cs}(W_f)$ along the chord endpoint $\mb{z}_0$.  
Since $e$ and $f$ are disjoint, the last statement of Lemma \ref{lem:WmoveisCslide} implies that $y_0 \neq z_0$.  However, it is still possible that some $y_i$ are equal to some $z_j$.  If no such equality exists,  it is easy to see that the corresponding relation in $\Cs$ is a multiple-chord version of the disjoint commutativity relation $D$, which again follows from the single chord version $D$ by induction.

We now must address the possible cases where some $y_i$ coincides with some $z_j$.  
We begin by noting that  the endpoints $\{\mb{y}_1, \ldots, \mb{y_k}\}$ and $\{\mb{z}_1, \ldots, \mb{z_m}\}$ of the  slid chords of $\mathfrak{Cs}(W_e)$ and $\mathfrak{Cs}(W_f)$ must be either disjoint or  nested with respect to the ordering $\prec$.  This follows  since  the slid chords correspond to the leaves of two subtrees of $T_G$ which themselves  must be either nested or disjoint.  (It is a general fact that  for any two oriented edges $\mb{y},\mb{z}\in\mathcal{E}_{or}(T_G)$ for which the  subtrees $T(\mb{y})$ and $T(\mb{z})$ do not contain the tail of $G$, the subtrees $T(\mb{y})$ and $T(\mb{z})$ are either nested or disjoint.)

Consider  the situation where the chord endpoints $\{\mb{y}_1, \ldots, \mb{y_k}\}$ and $\{\mb{z}_1, \ldots, \mb{z_m}\}$ are disjoint and $\mb{y}_i\neq \mb{\bar z}_0$ and $\mb{z}_j\neq \mb{\bar y}_0$ but some $y_i$ are equal to some $z_j$.  The simplest situation is where $k=m=1$ so that $\mb{y}_1=\mb{\bar z}_1$.  The commutativity relation in $\Pt$ then translates directly to the opposite end commutativity relation $O$ in $\Cs$.  The more general case $k,m\geq 1$  follows by induction using $I$, $D$, and $O$.   

Now we analyze the case where   $\{\mb{y}_1, \ldots, \mb{y_k}\}$ and $\{\mb{z}_1, \ldots, \mb{z_m}\}$ are disjoint but some $\mb{y}_i$ equals $\mb{\bar z}_0$ or some $\mb{z}_j$ equals $\mb{\bar y}_0$.  Note that Lemma \ref{lem:whichslides} implies that both relations cannot hold simultaneously, so we may assume, for instance, that $\mb{z}_j=\mb{\bar y}_0$. 
After the move $W_e$, we see that we are in the situation where the chord endpoints are nested, and we treat this case next. 

 Now assume that the endpoints of the slid chords of  $\mathfrak{Cs}(W_e)$ and $\mathfrak{Cs}(W_f)$ are  nested.   Again, we first treat a simple case.  Assume that $k=1$ and $m=2$, with  $\mb{y}_1=\mb{z}_2$, $\mb{y}_0=\mb{z}_1$ and  $\mb{z}_0\iprec\mb{z}_1\iprec\mb{z}_2$.  In this case, it is easy to see that the corresponding commutativity relation is given by the Left commutativity relation $L$.   
 If we instead had $\mb{z}_1\iprec\mb{z}_2\iprec\mb{z}_0$, then we would have obtained the right commutativity relation $R$.  Similarly, one can  verify that under the conditions $\mb{y}_1=\mb{z}_1$ and  $\mb{y}_0=\mb{z}_2$ then  one obtains  (the inverse of) relations $R$ and $L$ for $\mb{z}_0\iprec\mb{z}_1\iprec\mb{z}_2$ and $\mb{z}_1\iprec\mb{z}_2\iprec\mb{z}_0$  respectively.  As usual, the more general case of $k>1$ follows by induction, now requiring   $L$ or $R$ together with  $I$, $D$, and $O$.  

Finally, we must treat the pentagon relations where the edges $e$ and $f$ are adjacent.  There are twenty-four distinct case that we must analyze, depending on the order that the sectors adjacent to the two edges $e$ and $f$ are traversed.  We shall label these cases by the clockwise numbering of these sectors (up to cyclic permutation).  

We will treat only case 14325 in detail, which is depicted in figure \ref{fig:14325}.   If  the edge of $G$ separating sectors 2 and 5 is not  a generator,  it then corresponds to several, say $n$, chords of $C_G$.  In the figure, we have depicted the case $n=2$.   Note that bottom right corner of the pentagon in figure \ref{fig:14325} can be bypassed by applying the  $T$ relation.  After removing this chord diagram, only an $n$-chord version of the $A$ relation remains.  Since the $n$-chord version of $A$ can be proven by induction using only $A$ itself, 
the relation of case 14325 can be written as a product of $A$ and $T$. 

\begin{figure}[!h]
\begin{center}
\epsffile{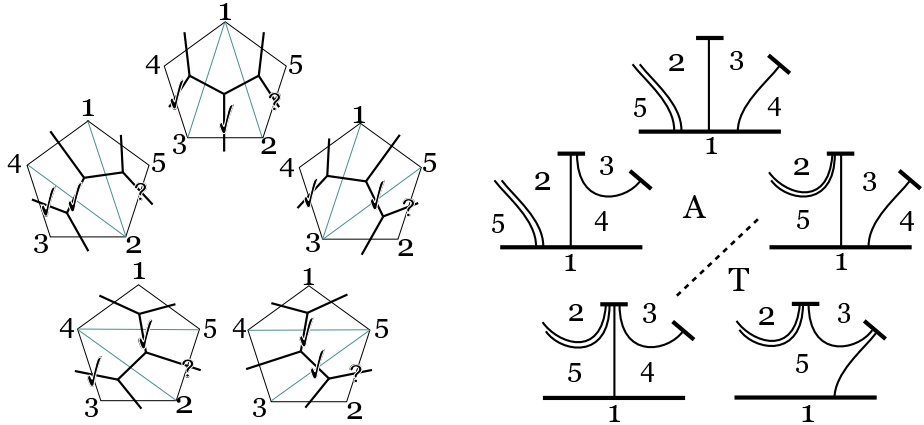}
\caption{Pentagon relation 14325.}
\label{fig:14325}
\end{center}
\end{figure}

Note that for any case whose label begins with $12$,  at least three out of the five moves are of type 1 or 2.  So in these cases, the only relation that can arise is involutivity $I$, as one can also check directly.  
We now  simply list the results of the other cases (we have not included the relations needed for the induction in the multiple-chord versions):
\vspace{.2in}

\begin{tabular}{l l | l l | l l}
\textrm{Case} & \textrm{follows from}  & \textrm{Case} & \textrm{follows from}& \textrm{Case} & \textrm{follows from}  \\
\hline 
13245 &  I  & 14235 & I  & 15234 & I \\
13254 & I   & 14253 & T, I & 15243 &  T, I \\
13425 &  A   & 14325 &  T, A & 15324 & T, A  \\
13452 &  I & 14352 & R  & 15342 & R \\
13524 & A & 14523 & A& 15423 & A \\
13542 & I    & 14532 & L& 15432 & L \\
\end{tabular}

\end{proof}

\section{Presentations}

As discussed in Section \ref{sect:markedborderedfatgraphs}, the mapping class group $\M1g$ has a presentation in terms of sequences of Whitehead moves on bordered fatgraphs.  By the results of the previous section, this presentation immediately leads to the following chord diagrammatic version.  

\begin{theorem}
The mapping class group $\M1g$  of a once bordered surface has an (infinite) presentation with generators given by sequences of chord slides on marked bordered chord diagrams such that the initial and final chord diagrams are isomorphic (as unmarked fatgraphs) and relations given by saying that two such sequences are equal if they differ by the insertion and deletion of any finite number of $T$, $I$, $C$, $L$, and $R$ relations.
\end{theorem}

As in the case of the Ptolemy groupoid, it is possible to consider quotients of $\Cs$ by any subgroup of the mapping class group $\M1g$ by considering them as full subgroupoids of the fundamental path groupoids of the corresponding  quotients of $|\mathcal{\hat G}_T|$.  
 In particular, consider the descending Johnson filtration $\M1g[k]\subset \M1g$ of $\M1g$ where $\M1g[k]$ is the subgroup acting trivially on the $k$th nilpotent quotient $N_k$ of $\pi_1$.   $\M1g[0]$ is the full mapping class group, and $\M1g[1]$ is the Torelli group, which is the subgroup of $\M1g$ acting trivially on the first homology $H_1(\S1g,\bZ)\cong H$ of $\S1g$.   Let $\mathfrak{Cs}_{g,1}[k]$ denote the quotient of $\Cs$ by $\M1g[k]$.  The groupoid $\mathfrak{Cs}_{g,1}[0]$ is naturally identified with the groupoid of chord slides on (unmarked) chord diagrams.  More generally,  $\mathfrak{Cs}_{g,1}[k]$ is the groupoid of chord slides on geometrically $N_k$-marked chord diagrams (see \cite{moritapenner} for a discussion of the fatgraph version), which we now describe.

Consider the composition of a $\pi_1$-marking of a fatgraph $G$ with the projection map $\pi_1\ra N_k$.  The result is a map $N_k\colon \mathcal{E}_{or}(G)\ra N_k$ called a geometric $N_k$-marking of $G$.  We shall focus on the particular case of $k=1$ where we obtain a map $H\colon \mathcal{E}_{or}(G)\ra H$ which we call a geometric $H$-marking of $G$. This map obviously satisfies the  abelian versions of the orientation, vertex, and surjectivity   conditions.    One can check that such a map also satisfies the following geometricity condition: for any two oriented edges $\mb{x}$ and $\mb{y}$ of $G$, the symplectic pairing $H(\mb{x})\cdot H(\mb{y})$ is equal to  $\langle \mb{x},\mb{y}\rangle$, where $\langle \; ,\; \rangle$ is a skew pairing  on $\mathcal{E}_{or}(G)$ defined by setting $\langle \mb{x},\mb{y}\rangle$ equal to $-1$ if up to cyclic permutation $\mb{x}<\mb{y}<\mb{\bar x}<\mb{\bar y}$, equal to $1$ if up to cyclic permutation $\mb{x}<\mb{y}<\mb{\bar x}<\mb{\bar y}$, and equal to $0$ otherwise (see \cite{bkp}).  

Since the generators and relations of $\Cs$ descend to generators and relations of each of the quotient groupoids, we immediately obtain the following theorems.  

\begin{theorem}
The groupoid $\mathfrak{Cs}_{g,1}[k]$ is generated by chord slides on $N_k$-marked bordered chord diagrams and has relations given by compositions of $T$, $I$, $C$, $L$, and $R$.  
\end{theorem}

\begin{theorem}
The mapping class group $\M1g$  of a once bordered surface has an (infinite) presentation with generators given by sequences of chord slides on (unmarked) bordered chord diagrams  beginning and ending at the same chord diagram,  and relations given by saying that two such sequences are equal if they differ by the insertion and deletion of any finite number of $T$, $I$, $C$, $L$, and $R$ relations.
\end{theorem}

\begin{theorem}
The Torelli  group $\mathcal{I}_{g,1}$  of a once bordered surface has a presentation with generators given by sequences of chord slides on geometrically $H$-marked bordered chord diagrams  beginning and ending at the same geometrically $H$-marked chord diagram  and relations given by saying that two such sequences are equal if they differ by the insertion and deletion of any finite number of $T$, $I$, $C$, $L$, and $R$ relations.
\end{theorem}

Finally, in \cite{moritapenner}, a notion of a finite presentation of an (infinite) groupoid was introduced and applied to the Torelli groupoid, which is the quotient of the Ptolemy groupoid $\Pt$ by the action of the Torelli group.  In particular, they used the action of the Torelli group $\mathcal{I}_{g,1}$ and the symplectic group $Sp(2g,\bZ)$, both of which are finitely generated, on a $\M1g$ fundamental domain of \Teich space  to give a finite description of the generators and relations, which they called a finite presentation,   of the Torelli groupoid.   An obvious modification of their proof to the setting of chord diagrams gives 
\begin{theorem}
The groupoid $\Cs[1]$ is finitely presentable in the sense of \cite{moritapenner}.
\end{theorem}

\section{Dual Chord Diagrams}

In this section we introduce the notion of a dual chord diagram.  For this, consider a chord diagram $C$ and the two linear orderings $<$ and $\prec$ of its oriented chords.  

\begin{definition}
The dual chord diagram to $C$, denoted $\widehat C$, is a chord diagram with the same set of oriented chords as $C$, but arranged so that the left to right order of chord endpoints along the core is determined by $<$.
\end{definition}
We comment here that this notion of duality is related to the duality which associates a trivalent fatgraph to a  triangulation of $\S1g$ based at $p$.  While this perspective is helpful, in particular in interpreting Observations \ref{obs:o1} and \ref{obs:o2}, it is not necessary for our goals, 
 so we leave it to  the interested reader to formulate the precise correspondence.  

Note that by the definition, the chords of a dual chord diagram $\widehat C$ under their preferred orientation with respect to $<$ are all oriented from left to right along the core.  
We will write $\mb{c}\ipred\mb{d}$ if the oriented chord $\mb{c}$ immediately precedes the oriented chord $\mb{d}$ under the linear ordering  $<$ restricted to the set of oriented chords.  We immediately have the following 
\begin{lemma}
For oriented chords $\mb{c}$ and $\mb{d}$, 
$\mb{c}\iprec \mb{d} $ if and only if $ \mb{\bar d}\ipred \mb{c}$.
\end{lemma} 
From this we see that the operation of taking duals  twice is not the identity, $\widehat{\widehat C}\neq C$.  The next lemma  follows immediately from the previous one.
\begin{lemma} For oriented chords $\mb{c}$ and $\mb{d}$ with $\mb{c}\iprec \mb{d}$, 
the chord slide of $\mb{c}$ along $\mb{d}$ in $C$ corresponds to the slide of $\mb{\bar d}$ along $\mb{c}$ in $\widehat C$.  Similarly, the chord slide of $\mb{d}$ along $\mb{c}$ in $C$ corresponds to the slide of $\mb{c}$ along $\mb{\bar d}$ in $\widehat C$.
\end{lemma}

\begin{proposition}
The dual chord diagram $\widehat C$ of a genus $g$ bordered chord diagram $C$ is again a genus $g$ bordered chord diagram.
\end{proposition}
\begin{proof}
The proposition can easily be verified for any one particular bordered chord diagram $C$.  Since the genus is preserved under chord slides, the proposition follows from the previous lemma and the fact that chord slides generate $\Cs$.
\end{proof}

\begin{observation}\label{obs:o1}
If $C$ is a marked bordered chord diagram,  we automatically obtain a map $\hat \pi_1\colon \mathcal{E}_{or}(X_{\widehat C})\ra \pi_1$  on its dual $\widehat C$ by setting $\hat \pi_1(\mb{x})=\pi_1(\mb{x})$.   We call this map a \emph{dual marking} of $\widehat C$.  Similarly, we define a dual geometric $H$ marking $\hat H\colon \mathcal{E}_{or}(X_{\widehat C})\ra H$ of $\widehat C$ by composing $\hat \pi_1$ with the abelianization map.  Dual markings behave rather differently than usual markings.  In fact, dual markings satisfy the orientation compatibility condition, but do not satisfy the vertex compatibility condition.   Also, contrary  to the usual situation, a chord slide of a chord end  $\mb{c}$ along $\mb{d}$ in $\widehat C$ alters the dual marking of the chord $c$ but leaves the dual marking of $d$ fixed.    
\end{observation}

\begin{theorem}
The mapping class group $\M1g$ has a presentation with generators given by sequences of chord slides on  genus $g$ dual-marked bordered chord diagrams beginning and ending at isomorphic diagrams and relations given by compositions of $T$, $I$, $C$, $L$, and $R$.
\end{theorem}
\begin{proof}
Using the previous  lemma, one can show that  the relations $T$, $I$, $C$, $L$, $R$, $O$, and $A$ are dual to $T$, $I$, $C$, $R$, $L$, $A$, and $O$ respectively.  The theorem then follows. 
\end{proof}

\begin{observation}\label{obs:o2}
One nice feature of dual chord diagrams not shared by ordinary chord diagrams is that the symplectic pairing can be easily read off from the diagram.  In particular,
$
\langle\mb{x},\mb{y}\rangle =-\sum_{p\in x\cap y} \varepsilon(\mb{x},\mb{y},p)
$, 
 where the sum is over all crossing points of the two chords $x$ and $y$, and  $\varepsilon(\mb{x},\mb{y},p)$ takes the values of plus or minus one, depending on whether or not the direction along $\mb{x}$ followed by the direction along  $\mb{y}$ gives an oriented basis for the plane at $p$. 
\end{observation}

\subsection{An integral algorithm}

In \cite{abp}, a simple  algorithm was provided to transform what was called a geometric basis for $H$ into a symplectic one. A priori, this algorithm only worked over the rational numbers, but here we give a proof that it is in fact integral.  We begin by recalling this algorithm.  

Recall that a geometric basis $X=\{X_1,\ldots,X_{2g}\}$ for $H$ is an ordered basis obtained from a geometrically $H$-marked bordered fatgraph by taking the set of $H$-markings of the generators,  $H(\mb{X}_G)=\{X_1,\ldots,X_{2g}\}$ with $X_i=H(\mb{x}_i)$ for $\mb{x}_i\in \mb{X}_G$.   
  Equivalently, one could take the dual $H$-markings of the chords of a dual bordered chord diagram, again ordered by $<$.
Such a basis has the property that all mutual intersection pairings are $\pm 1$ or zero.  

The algorithm in question builds a symplectic basis of the form $\{A_1,B_1,A_2,\ldots, B_g\}$ with $A_i\cdot B_j=\delta_{ij}$  out of $X=\{X_1,\ldots, X_{2g}\}$ by reiterating the following procedure.  Assume that the first $2k$ elements of  $X$ satisfy the requirements on their mutual intersection pairings.  Then let $A_{k+1}=X_{2k+1}$ and let $i>2k+1$ be minimal so that $X_{2k+1}\cdot X_{i}\neq 0$.  Then set  $B_{k+1}=X_i/(X_{2k+1}\cdot X_i)$ and rearrange the elements  so that $B_{k+1}$ comes immediately after $A_{k+1}$.  Finally,  modify the remaining basis elements by 
\begin{equation} \label{eq:finalstep}
X_j\mapsto X_j-(X_j\cdot B_{k+1}) A_{k+1} + (X_j\cdot A_{k+1}) B_{k+1}. 
\end{equation}

As already mentioned, this algorithm a priori only works over the rational numbers due to the step which divides $B_{k+1}$ by the  integer $(X_{2k+1}\cdot X_i)$.  We now give a dual chord slide interpretation of the algorithm which will show that it is in fact integral.

\begin{proposition}
The algorithm described above works over the integers.
\end{proposition}
\begin{proof}
Since intersection pairings of oriented edges of a chord diagram are always $\pm 1$ or zero, the proposition will follow once we show that each step in the algorithm can be obtained by performing a sequence of chord slides and  reversals of chord orientations.  As already noted,  the changing of sign of the $H$-marking of a chord is obtained by changing the orientation of the chord.

Now consider a dual marked  dual chord diagram $\widehat C$.  Let $\mb{c}$ and $\mb{d}$ be chords of $\widehat C$ that cross, meaning the pairing $\langle \mb{c},\mb{d}\rangle$ is non-zero.  Without loss of generality, assume that we have $\mb{c}<\mb{d}<\mb{\bar c}<\mb{\bar d}$ so that $\langle \mb{c},\mb{d}\rangle=-1$.  We can then consider the sequence of chord slides which slide all chord endpoints $\mb{x}$ with $\mb{c}<\mb{x}<\mb{\bar d}$  out of the region between $\mb{c}$ and $\mb{\bar d}$ as illustrated in  figure \ref{fig:sympline}.

\begin{figure}[!h]
\begin{center}
\epsffile{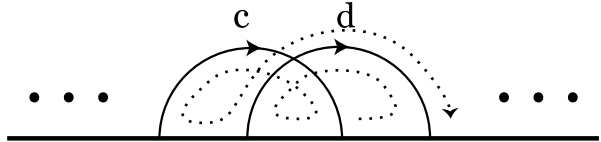}
\caption{Chord slides out of the region between  $\mb{c}$ and $\mb{\bar d}$.}
\label{fig:sympline}
\end{center}
\end{figure}

It is easy to see that the effect of this sequence of slides on the $\hat H$-marking  $\hat H(\mb{x})$ of \emph{any} oriented chord $\mb{x}$ is given by 
subtracting $\langle \mb{x}, \mb{c} \rangle\cdot  \hat H(\mb{d})$ and adding $ \langle \mb{x}, \mb{d} \rangle \cdot  \hat H(\mb{c})$. 
In particular, if we set $A_{k+1}=\hat H(\mb{c})$ and $B_{k+1}=\hat H(\mb{\bar d})$, so that $A_{k+1}\cdot B_{k+1} =1 $,  
then we recover Equation  \ref{eq:finalstep}.
Thus, we see that the final step of the algorithm can be obtained by chord slides.  

Finally, we have only to address the issue of reordering the generators.  But this issue is resolved by the observation  that any pair of chords  $\mb{c}$ and $\mb{d}$ of a dual chord diagram $\widehat C$ with $\mb{c}\ipred\mb{d}\ipred\mb{\bar c}\ipred \mb{\bar d}$ can be collectively repositioned  along the core of a diagram without changing the $H$-markings of any chord.   The proposition thus follows.
\end{proof}

{\it Acknowledgment: }
{\small The author owes a great debt to Jean-Baptiste Meilhan for many helpful discussions and comments, in particular for his recognition of Lemma 6.1, as well as  for his general encouragement in the writing up of these results. 
The author also gratefully thanks Robert Penner for similarly helpful discussions, comments, and encouragement.  }

\bibliographystyle{amsplain}

\begin{thebibliography}{999}

\bibitem{abp}
J. Andersen, A. Bene, R. Penner, 
{\it Groupoid lifts of mapping class representations for bordered surfaces},
preprint, arXiv: 0710.2651


\bibitem{bkp}
A. Bene, N. Kawazumi, R.C. Penner,
{\it Canonical liftls of the Johnson homomorphisms to the Torelli groupoid},
preprint, arXiv: 0707.2984

\bibitem{godin}
V. Godin,
{\it The unstable integral homology of the mapping class groups of a surface with boundary},
Math. Ann. 337 (2007), 15--60.

\bibitem{harer}
J. Harer, 
{\it The virtual cohomological dimension of the mapping
class group of an orientable surface},
Invent. Math. 84
(1986), 157--176.


\bibitem{moritapenner}
S. Morita, R.C. Penner
{\it Torelli groups, extended Johnson homomorphisms, and new cycles on the moduli space of curves},
to appear Math. Proc. Camb. Phil. Soc..

\bibitem{penner}
R. Penner,
{\it The decorated Teichm\"uller space of punctured surfaces},
Comm. Math. Phys. 113
(1987), 299--339.

\bibitem{Penner04}
---,
{\it Decorated Teichm\"uller theory of bordered surfaces},
Comm. Anal. Geom.   12 
(2004), 793--820.

\bibitem{Zieschang}
H. Zieschang,
{\it Surface and Planar Discontinuous Groups},
Lecture Notes in Mathematics 835,
Springer-Verlag 1980.


\end{thebibliography}

\end{document}